\documentclass[12pt]{article}
\usepackage[utf8]{inputenc}
\usepackage[T1]{fontenc}

\usepackage{amsmath, amsthm}
\usepackage{amssymb}

\newtheorem{theorem}{Theorem}[section]
\newtheorem{proposition}[theorem]{Proposition}

\newtheorem{lemma}[theorem]{Lemma}

\title{On Number of Rich Words}

\author{Josef Rukavicka\thanks{Department of Mathematics,
Faculty of Nuclear Sciences and Physical Engineering, CZECH TECHNICAL UNIVERSITY
IN PRAGUE
(josef.rukavicka@seznam.cz).}}

\newtheorem{definition}[theorem]{Definition}
\theoremstyle{remark}

\newtheorem{remark}[theorem]{Remark}

\date{\small{January 25, 2017}}

\begin{document}
\maketitle

\begin{abstract}
Any finite word $w$ of length $n$ contains at most $n+1$ distinct palindromic factors. If the bound  $n+1$ is reached, the word $w$ is called rich. 
The number of rich words of length $n$ over an alphabet of cardinality $q$ is denoted $R_n(q)$.   For binary alphabet,  Rubinchik and Shur deduced  that  ${R_n(2)}\leq c 1.605^n $ for some constant $c$. 
We prove that $\lim\limits_{n\rightarrow \infty }\sqrt[n]{R_n(q)}=1$ for any $q$, i.e. $R_n(q)$ has  a subexponential
growth on any alphabet. 
\end{abstract}

\section{Introduction}
The study of palindromes is a frequent topic and many diverse results may be found.
In recent years, some of the papers deal with so-called \emph{rich} words, or also words having \emph{palindromic defect $0$}.
They are words that have the maximum number of palindromic factors.
As noted by \cite{DrJuPi}, a finite word $w$ can contains at most $|w|+1$ distinct palindromic factors with $|w|$ being the length of $w$.
The rich words are exactly those that attain this bound. It is known that on binary alphabet the set of rich words contains factors of Sturmian words, factors of complementary symmetric Rote words, factors of the period-doubling word, etc., see \cite{DrJuPi,BlBrLaVu11,Ba_phd,ScSh16}. On multiliteral alphabet, the set of rich words contains for example factors of Arnoux--Rauzy words and factors of words coding symmetric interval exchange.

Rich words can be characterized using various properties, see for instance \cite{GlJuWiZa,BuLuGlZa2,BaPeSta2}.
The concept of rich words can also be generalized to respect so-called pseudopalindromes, see \cite{PeSta2}.
In this paper we focus on an unsolved question of computing the number of rich words of length $n$ over an alphabet with $q>1$ letters. This number is denoted  $R_n(q)$. 

This question is investigated in \cite{Vesti2014}, where J. Vesti gives a recursive lower bound on the number of rich words of length $n$, and an upper bound on the number of binary rich words.
Both these estimates seem to be very rough.
In \cite{GuShSh15},  C. Guo, J. Shallit and A.M. Shur  constructed for each $n$  a large set of rich words of length $n$. Their construction gives,  currently,  the best lower bound on the number of binary rich words, namely 
$R_n(2)\geq \frac{C^{\sqrt{n}}}{p(n)}$, 
 where $p(n)$ is a polynomial and the constant $C \approx 37$. On the other hand, the best known upper bound is exponential.  As mentioned in \cite{GuShSh15},  calculation performed recently by M. Rubinchik provides the upper bound  $  R_n(2)\leq c 1.605^n$ for some constant $c$, see \cite{RuSh15}.

Our main result stated as Theorem  \ref{label_th_F} shows that $R_n(q)$ has  a subexponential
growth on any alphabet.  More precisely,   we prove  
$$\lim\limits_{n\rightarrow \infty}\sqrt[n]{R_n(q)}=1\,.$$
In \cite{Shur2012}, Shur calls languages with the above property small.
Our result  is   an argument in favor of a conjecture formulated in \cite{GuShSh15}  saying that  for some infinitely growing function $g(n)$ the following holds true  ${R_n(2)} = \mathcal{O} \Bigl(\frac{n}{g(n)}\Bigr)^{\sqrt{n}}$ .

To derive our result we consider a specific factorization of  a rich word into distinct rich palindromes, here called UPS-factorization (Unioccurrent Palindromic Suffix factorization),  see Definition  \ref{factorization}.  
Let us mention that  another palindromic factorizations have already been studied, see \cite{Bannai2015,FrPuZa}: Minimal (minimal number of palindromes), maximal (every palindrome cannot be extended on the given position) and diverse (all palindromes are distinct). Note that only the minimal palindromic factorization has to exist for every word.

The article is organized as follows: Section \ref{notation} recalls notation and known results. In Section \ref{faktor} we study a relevant property of  UPS-factorization. The last section is devoted to the proof of our main result.

\section{Preliminaries}\label{notation}
\noindent
Let us start with a couple of definitions:
Let $A$ be an alphabet of $q$ letters, where $q>1$ and $q\in \mathbb{N}$ ($\mathbb{N}$ denotes the set of nonnegative integers).
A finite sequence $u_1u_2\cdots u_n$ with $u_i \in A$ is a \emph{finite word}.
Its length is $n$ and is denoted $|u_1u_2\cdots u_n| = n$.
Let $A^n$ denote the set of words of length $n$. We define that $A^0$ contains just the empty word.
It is clear that the size of $A^n$ is equal to $q^n$. \\
Given $u=u_1u_2\cdots u_n \in A^n$ and $v=v_1v_2\cdots v_k \in A^k$ with $0\leq k \leq n$, we say that $v$ is a \emph{factor} of $u$ if there exists $i$ such that $0<i$, $i+k \leq n$ and $u_i=v_1$, $u_{i+1}=v_2$, $\dots$, $u_{i+k-1}=v_k$.
 \\
A word $u=u_1u_2\cdots u_n$ is called a \emph{palindrome} if $u_1u_2\cdots u_n=u_nu_{n-1}\cdots u_1$. The empty word is considered to be a palindrome and a factor of any word.

A word $u$ of length $n$  is called \emph{rich} if $u$ has $n+1$ distinct  palindromic factors. Clearly,  $u=u_1u_2\cdots u_n$  is rich if and only if its \emph{reversal} $ u_nu_{n-1}\cdots u_1$ is rich as well. \\

Any factor of a rich word is rich as well, see \cite{GlJuWiZa}. In other words,  the language of rich words is factorial. In particular it means that 
$R_n(q)R_m(q)\leq R_{n+m}(q)$ for any $m, n, q \in \mathbb{N}$. Therefore,  the Fekete's lemma  implies existence of the limit of  $\sqrt[n]{R_n(q)}$ and moreover 
\[
\lim\limits_{n\rightarrow \infty}\sqrt[n]{R_n(q)}= \inf \left \{\sqrt[n]{R_n(q)} \colon n \in \mathbb{N} \right\}.
\]
For a fixed $n_0$,   one can find the number of all rich words of length $n_0$ and obtain an upper bound on the limit. 
Using computer Rubinchik counted $R_n(2)$ for $n\leq 60$, (see the sequence A216264 in OEIS). As  $ \sqrt[60]{R_{60}(2)} < 1.605$, he obtained the  upper bound given in Introduction. 

\noindent 

As shown in \cite{GlJuWiZa},  any rich word $u$ over alphabet $A$ is richly prolongable, i.e., there exist  letters $a, b \in A$ such that $aub$ is also rich. Thus a rich word is a factor of an arbitrarily long rich word.  But the question whether two  rich words  can appear simultaneously as factors of a longer rich word may have  negative answer. It means that  the language of rich words is not recurrent.  This fact makes enumeration  of rich words hard.

\section{Factorization of rich words into rich palindromes}\label{faktor}
Let us recall one important property of rich words \cite[Definition $4$ and Proposition $3$]{DrJuPi}: the longest palindromic suffix of a rich word $w$ has exactly one occurrence in $w$ (we say that the longest palindromic suffix of $w$ is \emph{unioccurrent} in $w$).
It implies that $w=w^{(1)}w_1$, where $w_1$ is a palindrome which is not a factor of $w^{(1)}$. Since every factor of a rich word is a rich word as well, it follows that  $w^{(1)}$ is a rich word and thus $w^{(1)}=w^{(2)}w_2$, where $w_2$ is a palindrome which is not a factor of $w^{(2)}$.  Obviously $w_1\not =w_2$.  We can repeat the process until $w^{(p)}$ is the empty word for some $p\in \mathbb{N}$, $p\geq 1$. We express these ideas by the following lemma:

\begin{lemma}
\label{lemma_A}
Let $w$ be a rich word. There exist distinct non-empty palindromes $w_1,w_2,\dots,w_p$ such that
\begin{multline} \label{eqtn_001}
w=w_pw_{p-1}\cdots w_2w_1
 \mbox{ and }w_i \mbox{ is the longest palindromic suffix of }\\ w_pw_{p-1}\cdots w_i \mbox{ for }i=1,2,\dots ,p.
\end{multline}
\end{lemma}
\begin{definition}\label{factorization}
We define UPS-factorization (Unioccurrent Palindromic Suffix factorization) to be the factorization of a rich word $w$ into the form (\ref{eqtn_001}).
\end{definition}
Since $w_i$ in the factorization (\ref{eqtn_001}) are non-empty, it is clear that $p\leq n=\vert w\vert$. From the fact that the palindromes $w_i$ in  the factorization (\ref{eqtn_001}) are distinct we can derive a better upper bound for $p$. The aim of this section is to prove the following theorem:

\begin{theorem}
\label{theorem_A}
There is a constant $c>1$ such that for any rich word $w$ of length $n$ the number of palindromes in the UPS-factorization of $w=w_pw_{p-1}\cdots w_2w_1$ satisfies 
\begin{equation}
\label{eqtn_002}
p\leq c\frac{n}{\ln{n}}\mbox{.}
\end{equation}
\end{theorem}
Before proving the theorem, we show two auxiliary lemmas:
\begin{lemma}
\label{lemma_B}
Let $q,n,t\in \mathbb{N}$ such that
\begin{equation}
\label{eqtn_003}
\sum_{i=1}^tiq^{\lceil\frac{i}{2}\rceil}\geq n\mbox{.}
\end{equation} The number $p$ of palindromes in the UPS-factorization $w=w_pw_{p-1}\dots w_2w_1$ of any rich word $w$ with $n=\vert w\vert$ satisfies
\begin{equation}
\label{eqtn_004}
p\leq \sum_{i=1}^tq^{\lceil\frac{i}{2}\rceil}\mbox{.}
\end{equation}
\end{lemma}
\begin{proof}
Let $f_1,f_2,f_3,\dots$ be an infinite sequence of all non-empty palindromes over an alphabet $A$ with $q=\vert A\vert$ letters, where the palindromes are ordered in such a way that $i<j$ implies that $\vert f_i\vert \leq \vert f_j\vert$.
In consequence $f_1,\dots,f_q$ are palindromes of length $1$, $f_{q+1}, \dots,f_{2q}$ are palindromes of length $2$, etc.
%Since $w_1,\dots,w_p$ are distinct non-empty palindromes, it follows that the map determined by $f_j\mapsto w_i$ is injective for $i=1,\dots, p$.
Since $w_1,\dots,w_p$ are distinct non-empty palindromes we have $\sum_{i=1}^p\vert f_i\vert\leq \sum_{i=1}^p\vert w_i\vert=n$.
The number of palindromes of length $i$ over the alphabet $A$ with $q$ letters is equal to $q^{\lceil\frac{i}{2}\rceil}$ (just consider that that the ``first half'' of a palindrome determines the second half).
The number $\sum_{i=1}^tiq^{\lceil\frac{i}{2}\rceil}$ equals the length of a word concatenated from all palindromes of length less than or equal to $t$.
Since $\sum_{i=1}^p \vert f_i\vert \leq n \leq \sum_{i=1}^tiq^{\lceil\frac{i}{2}\rceil}$, it follows that the number of palindromes $p$ is less than or equal to the number of all palindromes of length at most $t$; this explains the inequality (\ref{eqtn_004}).
\end{proof}
\begin{lemma}
\label{lemma_C}
Let $N\in \mathbb{N}$, $x\in \mathbb{R}$, $x>1$ such that $N(x-1)\geq 2$. We have
\begin{equation}
\label{eqtn_005}
\frac{N x^{N}}{2(x-1)}\leq \sum_{i=1}^{N}ix^{i-1}\leq \frac{N x^{N}}{(x-1)}\mbox{.}
\end{equation}
\end{lemma}
\begin{proof}
The sum of the first $N$ terms of a geometric series with the quotient $x$ is equal to $\sum_{i=1}^{N}x^i=\frac{x^{N+1}-x}{x-1}$. Taking the derivative of this formula with respect to $x$ with $x>1$ we obtain:
$\sum_{i=1}^{N}ix^{i-1}=\frac{x^{N}(N(x-1)-1)+1}{(x-1)^2}=\frac{N x^{N}}{x-1}+\frac{1-x^{N}}{(x-1)^2}$.
It follows that the right inequality of (\ref{eqtn_005}) holds for all $N\in \mathbb{N}$ and $x>1$. The condition $N(x-1)\geq 2$ implies that $\frac{1}{2}N(x-1)\leq N(x-1)-1$, which explains the left inequality of (\ref{eqtn_005}).
\end{proof}
We can start the proof of Theorem \ref{theorem_A}:
\begin{proof}[Proof of Theorem \ref{theorem_A}]
Let $t\in \mathbb{N}$ be a minimal nonnegative integer such that the inequality (\ref{eqtn_003}) in Lemma \ref{lemma_B} holds. It means that:
\begin{equation}
\label{eqtn_006}
n>\sum_{i=1}^{t-1}iq^{\lceil\frac{i}{2}\rceil}\geq\sum_{i=1}^{t-1}iq^{\frac{i}{2}}=q^{\frac{1}{2}}\sum_{i=1}^{t-1}iq^{\frac{i-1}{2}}\geq \frac{(t-1)q^{\frac{t}{2}}}{2(q^{\frac{1}{2}}-1)}\mbox{,}
\end{equation}
where for the last inequality we exploited (\ref{eqtn_005}) with $N=t-1$ and $x=q^{\frac{1}{2}}$. If $q\geq 9$, then the condition $N(x-1)=(t-1)(q^{\frac{1}{2}}-1)\geq 2$ is fulfilled (it is the condition from Lemma \ref{lemma_C}) for any $t\geq 2$. Hence let us suppose that $q\geq 9$ and $t\geq 2$. From (\ref{eqtn_006}) we obtain:
\begin{equation}
\label{eqtn_007}
\frac{q^{\frac{t}{2}}}{q^{\frac{1}{2}}-1}\leq \frac{2n}{t-1}\leq \frac{4n}{t}\mbox{.}
\end{equation}
Since $t$ is such that the inequality (\ref{eqtn_003}) holds and $i\leq q^{\frac{i+1}{2}}$ for any $i\in \mathbb{N}$ and $q\geq 2$, we can write:
\begin{equation}
\label{eqtn_008}
n\leq \sum_{i=1}^t iq^{\frac{i+1}{2}}\leq \sum_{i=1}^t q^{i+1}=q^2\frac{q^t-1}{q-1}\leq\frac{q^2}{q-1}q^t\leq q^{2t}\mbox{.}
\end{equation}
We apply a logarithm on the previous inequality:
\begin{equation}
\label{eqtn_009}
\ln{n}\leq 2t\ln{q}\mbox{.}
\end{equation}
An upper bound for the number of palindromes $p$ in UPS-factorization follows from (\ref{eqtn_004}), (\ref{eqtn_007}), and (\ref{eqtn_009}):
\begin{equation}
\label{eqtn_10}
p\leq \sum_{i=1}^tq^{\lceil\frac{i}{2}\rceil}\leq \sum_{i=1}^tq^{\frac{i+1}{2}}\leq q^{\frac{3}{2}}\frac{q^{\frac{t}{2}}}{q^{\frac{1}{2}}-1} \leq q^{\frac{3}{2}}\frac{4n}{t}\leq q^{\frac{3}{2}}8\ln{q}\frac{n}{\ln{n}}\mbox{.}
\end{equation}
The previous inequality supposes that $q\geq 9$ and $t\geq 2$. If $t=1$ then we can easily derive from (\ref{eqtn_003}) that $n\leq q$ and consequently $p\leq n\leq q$. Thus the inequality $p\leq q^{\frac{3}{2}}8\ln{q}\frac{n}{\ln{n}}$ holds as well for this case.
Since every rich word over an alphabet with the cardinality $q<9$ is also a rich word over the alphabet with the cardinality $9$, the estimate (\ref{eqtn_002}) in Theorem \ref{theorem_A} holds if we set the constant $c$ as follows: $c=\max\{8q^{\frac{3}{2}}\ln{q}, 8 \cdot 9^{\frac{3}{2}}\ln{9}\}$.
\end{proof}

\begin{remark}
Theorem~\ref{theorem_A} implies that average length of a palindrome of UPS-factorization of a rich word of length $n$ is $\mathcal{O}(\ln(n))$.
Note that in \cite{RuSh16} it is shown that most of palindromic factors of a random word of length $n$ are of length close to $\ln(n)$.
\end{remark}

\section{Rich words form a small language}
The aim of this section is to show that the set of rich words forms a small language, see Theorem \ref{label_th_F}. \\

We present a recurrent inequality for $R_n(q)$. To ease our notation we omit the specification of the cardinality of alphabet and write 
$R_n$ instead of $R_n(q)$.

Denote $\kappa_n= \left \lceil c\frac{n}{\ln{n}} \right \rceil$, where $c$ is the constant from Theorem \ref{theorem_A} and $n\geq 2$.

\begin{theorem}
\label{label_th_D}
Let $n\geq 2$, then
\begin{equation}
\label{eqtn_11}
R_n\leq \sum_{p=1}^{\kappa_n}\sum_{\substack{n_1+n_2+\dots +n_p=n \\ n_1,n_2,\dots, n_p\geq 1}}R_{\lceil\frac{n_1}{2}\rceil}R_{\lceil\frac{n_2}{2}\rceil}\dots R_{\lceil\frac{n_p}{2}\rceil}\mbox{.}
\end{equation}
\end{theorem}
\begin{proof}
Given $p,n_1,n_2,\dots,n_p$, let $R(n_1,n_2,\dots, n_p)$ denote the number of rich words with UPS-factorization $w=w_pw_{p-1}\dots w_1$, where $\vert w_i\vert=n_i$ for $i=1,2,\dots,p$.
Note that any palindrome $w_i$ is uniquely determined by its prefix of length $\lceil\frac{n_i}{2}\rceil$; obviously this prefix is rich. Hence the number of words that appears in UPS-factorization as $w_i$ cannot be larger than $R_{\lceil\frac{n_i}{2}\rceil}$. It follows that $R(n_p,n_{p-1},\dots, n_1)\leq R_{\lceil\frac{n_1}{2}\rceil}R_{\lceil\frac{n_2}{2}\rceil}\dots R_{\lceil\frac{n_p}{2}\rceil}$. The sum of this result over all possible $p$ (see Theorem \ref{theorem_A}) and $n_1,n_2,\dots,n_p$ completes the proof.
\end{proof}

\begin{proposition}
\label{label_th_E}
If $h>1,K\geq 1$ such that $R_n\leq Kh^n$ for all $n$, then $\lim\limits_{n\rightarrow \infty} \sqrt[n]{R_n}\leq \sqrt{h}$.
\end{proposition}
\begin{proof}
For any integers $p,n_1,\dots,n_p\geq 1$, the assumption implies  that\\
$R_{\lceil\frac{n_1}{2}\rceil}R_{\lceil\frac{n_2}{2}\rceil}\dots R_{\lceil\frac{n_p}{2}\rceil}\leq K^ph^{\frac{n_1+1}{2}}h^{\frac{n_2+1}{2}}\dots h^{\frac{n_p+1}{2}} \leq K^ph^{\frac{n+p}{2}}$.
Exploiting (\ref{eqtn_11}) we obtain:
\begin{equation}
\label{eqtn_12}
R_n\leq K^{\kappa_n}h^{\frac{n+\kappa_n}{2}}\sum_{p=1}^{\kappa_n}\sum_{\substack{n_1+n_2+\dots +n_p=n \\ n_1,n_2,\dots, n_p\geq 1}} 1\mbox{.}
\end{equation}
The sum $$S_n=\sum_{\substack{n_1+n_2+\dots +n_p=n \\ n_1,n_2,\dots, n_p\geq 1}}1$$ can be interpreted as the number of ways how to distribute $n$ coins between $p$ people in such a way that everyone has at least one coin. That is why $S_n=\binom{n-1}{p-1}$.\\
It is known (see Appendix for the proof) that
\begin{equation}
\label{eqtn_13}
\sum_{i=0}^L\binom{N}{i}\leq \left(\frac{eN}{L}\right)^L\mbox{, for any }L,N\in \mathbb{N} \mbox{ and } L\leq N\mbox{.}
\end{equation}
From (\ref{eqtn_12}) we can write:
$R_n\leq K^{\kappa_n}h^{\frac{n+\kappa_n}{2}}\binom{en}{\kappa_n}^{\kappa_n}$. To evaluate $\sqrt[n]{R_n}$, just recall that $\lim\limits_{n\rightarrow \infty}(const)^{\frac{\kappa_n}{n}}=\lim\limits_{n\rightarrow \infty}(const)^{\frac{c}{\ln{n}}}=1$ for any constant $const$ and moreover $\lim\limits_{n\rightarrow \infty}\left(\frac{n}{\kappa_n}\right)^{\frac{\kappa_n}{n}}=\lim\limits_{n\rightarrow \infty}(c\ln{n})^{\frac{1}{c\ln{n}}}=1$.
\end{proof}

The main theorem of this paper is a simple consequence of the previous proposition.
\begin{theorem}
\label{label_th_F}
Let $R_n$ denote the number of rich words of length $n$ over an alphabet with $q$ letters.
We have $\lim\limits_{n\rightarrow \infty}\sqrt[n]{R_n}=1$.
\end{theorem}
\begin{proof}
Let us suppose that $\lim_{n\rightarrow \infty}\sqrt[n]{R_n}=\lambda>1$. We are going to find $\epsilon>0$ such that $\lambda+\epsilon<\lambda^2$. The definition of a limit implies that there is $n_0$ such that $\sqrt[n]{R_n}<\lambda+\epsilon$ for any $n>n_0$, i.e. $R_n<(\lambda+\epsilon)^n$. Let $K=\max\{R_1,R_2,\dots,R_{n_0}\}$. It holds for any $n\in \mathbb{N}$ that $R_n\leq K(\lambda+\epsilon)^n$. Using Proposition \ref{label_th_E} we obtain $\lim\limits_{n\rightarrow \infty}\sqrt[n]{R_n}\leq \sqrt{\lambda+\epsilon}<\lambda$, and this is a contradiction to our assumption that $\lim\limits_{n\rightarrow \infty}\sqrt[n]{R_n}=\lambda>1$.
\end{proof}

\section{Appendix}
For the reader's convenience, we provide a proof of the well-known inequality we used the proof of Proposition \ref{label_th_E}.
\begin{lemma}
$\sum_{k=0}^L\binom{N}{k}\leq \left(\frac{eN}{L}\right)^L$, where $L\leq N$ and $L,N\in \mathbb{N}$.
\end{lemma}
\begin{proof}
Consider $x\in (0,1]$. The binomial theorem states that $$(1+x)^N=\sum_{k=0}^N\binom{N}{k}x^k\geq \sum_{k=0}^L\binom{N}{k}x^k\mbox{.}$$
By dividing by the factor $x^L$ we obtain
$$ \sum_{k=0}^L\binom{N}{k}x^{k-L}\leq \frac{(1+x)^N}{x^L}\mbox{.}$$
Since $x\in (0,1]$ and $k-L\leq 0$, then $x^{k-L}\geq 1$, it follows that
$$ \sum_{k=0}^L\binom{N}{k}\leq \frac{(1+x)^N}{x^L}\mbox{.}$$
Let us substitute $x=\frac{L}{N}\in (0,1]$ and let us exploit the inequality $1+x<e^x$, that holds for all $x>0$:
$$ \frac{(1+x)^N}{x^L}\leq \frac{e^{xN}}{x^L}=\frac{e^{\frac{L}{N}N}}{(\frac{L}{N})^L}=\left(\frac{eN}{L}\right)^L \mbox{.}$$
\end{proof}

\section*{Acknowledgments}

The author wishes to thank Edita Pelantová and Štěpán Starosta for their useful comments. 
The authors acknowledges support by the Czech Science
Foundation grant GA\v CR 13-03538S and by the Grant Agency of the Czech Technical University in Prague, grant No. SGS14/205/OHK4/3T/14.

\bibliographystyle{siam}
\IfFileExists{biblio.bib}{\bibliography{biblio}}{\bibliography{../!bibliography/biblio}}

\begin{thebibliography}{10}

\bibitem{Ba_phd}
{\sc L.~Balkov\'a}, {\em Beta-integers and Quasicrystals}, PhD thesis, Czech
  Technical University in Prague and Université Paris Diderot-Paris 7, 2008.

\bibitem{BaPeSta2}
{\sc L.~Balkov\'a, E.~Pelantov\'a, and {\v{S}}.~Starosta}, {\em Sturmian jungle
  (or garden?) on multiliteral alphabets}, RAIRO-Theor. Inf. Appl., 44 (2010),
  pp.~443--470.

\bibitem{Bannai2015}
{\sc H.~Bannai, T.~Gagie, S.~Inenaga, J.~K{\"a}rkk{\"a}inen, D.~Kempa,
  M.~Pi{ą}tkowski, S.~J. Puglisi, and S.~Sugimoto}, {\em Diverse palindromic
  factorization is {NP}-complete}, in Developments in Language Theory: 19th
  International Conference, DLT 2015, Liverpool, UK, July 27-30, 2015,
  Proceedings., I.~Potapov, ed., Springer International Publishing, 2015,
  pp.~85--96.

\bibitem{BlBrLaVu11}
{\sc A.~Blondin~Mass\'{e}, S.~Brlek, S.~Labb\'e, and L.~Vuillon}, {\em
  Palindromic complexity of codings of rotations}, Theor. Comput. Sci., 412
  (2011), pp.~6455--6463.

\bibitem{BuLuGlZa2}
{\sc M.~Bucci, A.~{De Luca}, A.~Glen, and L.~Q. Zamboni}, {\em A new
  characteristic property of rich words}, Theor. Comput. Sci., 410 (2009),
  pp.~2860--2863.

\bibitem{DrJuPi}
{\sc X.~Droubay, J.~Justin, and G.~Pirillo}, {\em Episturmian words and some
  constructions of de {L}uca and {R}auzy}, Theor. Comput. Sci., 255 (2001),
  pp.~539--553.

\bibitem{FrPuZa}
{\sc A.~Frid, S.~Puzynina, and L.~Zamboni}, {\em On palindromic factorization
  of words}, Adv. Appl. Math., 50 (2013), pp.~737--748.

\bibitem{GlJuWiZa}
{\sc A.~Glen, J.~Justin, S.~Widmer, and L.~Q. Zamboni}, {\em Palindromic
  richness}, Eur. J. Combin., 30 (2009), pp.~510--531.

\bibitem{GuShSh15}
{\sc C.~Guo, J.~Shallit, and A.~M. Shur}, {\em Palindromic rich words and
  run-length encodings}, Inform. Process. Lett., 116 (2016), pp.~735--738.

\bibitem{PeSta2}
{\sc E.~Pelantov\'a and {\v{S}}.~Starosta}, {\em Palindromic richness for
  languages invariant under more symmetries}, Theor. Comput. Sci, 518 (2014),
  pp.~42--63.

\bibitem{RuSh15}
{\sc M.~Rubinchik and A.~M. Shur}, {\em EERTREE: An Efficient Data Structure
  for Processing Palindromes in Strings}, Springer International Publishing,
  Cham, 2016, pp.~321--333.

\bibitem{RuSh16}
{\sc M.~Rubinchik and A.~M. Shur}, {\em The number of distinct subpalindromes
  in random words}, Fund. Inform., 145 (2016), pp.~371--384.

\bibitem{ScSh16}
{\sc L.~Schaeffer and J.~Shallit}, {\em Closed, palindromic, rich, privileged,
  trapezoidal, and balanced words in automatic sequences}, Electr. J. Comb., 23
  (2016), p.~P1.25.

\bibitem{Shur2012}
{\sc A.~M. Shur}, {\em Growth properties of power-free languages}, Computer
  Science Review, 6 (2012), pp.~187--208.

\bibitem{Vesti2014}
{\sc J.~Vesti}, {\em Extensions of rich words}, Theor. Comput. Sci., 548
  (2014), pp.~14--24.

\end{thebibliography}

\end{document}